\documentclass[11pt, a4paper, reqno, english]{amsart}

\usepackage{hyperref}
\usepackage[all]{xy}
\usepackage{amssymb}

\newcommand\CC{\mathbb{C}}
\newcommand\FF{\mathbb{F}}
\newcommand\PP{\mathbb{P}}
\newcommand\Pone{{\PP^1}}
\newcommand\Ptwo{{\PP^2}}
\newcommand\ZZ{\mathbb{Z}}
\newcommand\QQ{\mathbb{Q}}
\newcommand\RR{\mathbb{R}}
\newcommand\GG{\mathbb{G}}
\newcommand\Ga{\GG_\mathrm{a}}
\newcommand\Gm{\GG_\mathrm{m}}
\newcommand\kbar{{\overline{k}}}

\renewcommand\O{\mathcal{O}}

\newcommand{\Aone}{{\mathbf A}_1}
\newcommand{\Atwo}{{\mathbf A}_2}
\newcommand{\Athree}{{\mathbf A}_3}
\newcommand{\Afour}{{\mathbf A}_4}

\newcommand{\Dfour}{{\mathbf D}_4}
\newcommand{\Dfive}{{\mathbf D}_5}
\newcommand{\Esix}{{\mathbf E}_6}
\newcommand{\Eseven}{{\mathbf E}_7}
\newcommand{\Eeight}{{\mathbf E}_8}
\newcommand{\tS}{{\widetilde S}}
\newcommand\inj{\hookrightarrow}
\newcommand\xx{\mathbf{x}}

\DeclareMathOperator{\rk}{rk}
\DeclareMathOperator{\Pic}{Pic}
\DeclareMathOperator\Aut{Aut}
\DeclareMathOperator\Bl{Bl}
\DeclareMathOperator\PGL{PGL}
\DeclareMathOperator\Gal{Gal}

\newcommand\rto{\dashrightarrow}
\newcommand\dpbox[2]{#2}
\newcommand\ecbox[2]{*+[F]{#2}}

\newtheorem*{theorem*}{Theorem}
\newtheorem{lemma}{Lemma}

\begin{document}

\title[Del Pezzo surfaces that are equivariant compactifications]{Singular del
  Pezzo surfaces that are equivariant compactifications}

\author{Ulrich Derenthal}

\address{Mathematisches Institut, Albert-Ludwigs-Universit\"at Freiburg,
  Eckerstr. 1, 79104 Freiburg, Germany}

\email{ulrich.derenthal@math.uni-freiburg.de}

\author{Daniel Loughran}

\address{Department of Mathematics, University Walk, Bristol, UK, BS8 1TW}

\email{daniel.loughran@bristol.ac.uk}

\begin{abstract}
  We determine which singular del Pezzo surfaces are equivariant
  compactifications of $\Ga^2$, to assist with proofs of Manin's conjecture
  for such surfaces. Additionally, we give an example of a singular quartic
  del Pezzo surface that is an equivariant compactification of $\Ga \rtimes
  \Gm$.
\end{abstract}



\maketitle

\section{Introduction}

Let $X \subset \PP^n$ be a projective algebraic variety defined over
the field $\QQ$ of rational numbers. If $X$ contains infinitely many
rational points, one is interested in the asymptotic behaviour of
the number of rational points of bounded height. More precisely, for
a point $\xx \in X(\QQ)$ given by primitive integral coordinates
$(x_0, \dots, x_n)$, the \emph{height} is defined as $H(\xx) =
\max\{|x_0|, \dots, |x_n|\}$. As rational points may
\emph{accumulate} on closed subvarieties of $X$, we are interested
in the counting function
\[N_U(B) = \#\{\xx \in U(\QQ) \mid H(\xx) \le B\}\] for suitable
open subsets $U$ of $X$.

A conjecture of Manin \cite{MR89m:11060} predicts the asymptotic
behaviour of $N_U(B)$ precisely for a large class of varieties. In
recent years, Manin's conjecture has received attention especially
in dimension $2$, where it is expected to hold for (possibly
singular) del Pezzo surfaces.

Recall that del Pezzo surfaces are classically defined as
non-singular projective surfaces whose anticanonical class is ample;
in order to distinguish them from the objects defined next, we will
call them \emph{ordinary del Pezzo surfaces}.  A \emph{singular del
Pezzo surface} is a singular projective normal surface with only
$\mathbf{ADE}$-singularities, and whose anticanonical class is
ample. A \emph{generalised del Pezzo surface} is either an ordinary
del Pezzo surface, or a minimal desingularisation of a singular del
Pezzo surface.

Most proofs of Manin's conjecture fall into two cases:
\begin{itemize}
\item For varieties that are \emph{equivariant compactifications} of certain
  algebraic groups (see Section~\ref{sec:preliminary} for details), one
  may apply techniques of \emph{harmonic analysis on adelic groups}. In
  particular, this has led to the proof of Manin's conjecture for all toric
  varieties \cite{MR1620682} and equivariant compactifications of vector
  spaces \cite{MR1906155}.
\item Without using such a structure, Manin's conjecture has been proved in
  some cases via \emph{universal torsors}. This goes back to Salberger
  \cite{MR1679841}. Here, one parameterises the rational points on $X$ by
  integral points on certain higher-dimensional varieties, called universal
  torsors, which turn out to be easier to count.
\end{itemize}

To identify del Pezzo surfaces for which proving Manin's conjecture
using universal torsors is worthwhile, one should know in advance
which ones are covered by more general results such as
\cite{MR1620682} and \cite{MR1906155}.

Toric del Pezzo surfaces (i.e., del Pezzo surfaces which are
equivariant compactifications of the two-dimensional torus $\Gm^2$)
have been classified: ordinary del Pezzo surfaces are toric
precisely in degree $\ge 6$. In lower degrees, there are some toric
singular del Pezzo surfaces, for example a cubic surface with
$3\Atwo$ singularities, for which Manin's conjecture was proved not
only by the general results of \cite{MR1620682}, \cite{MR1679841},
but also by more direct methods in \cite{MR2000b:11075},
\cite{MR2000b:11074}, \cite{MR2000f:11080}.  The classification of
all toric singular del Pezzo surfaces is known and can be found in
\cite{math.AG/0604194}, for example.

The purpose of this note is to identify all del Pezzo surfaces that
are $\Ga^2$-varieties (i.e., equivariant compactifications of the
two-dimensional additive group $\Ga^2$), so that Manin's conjecture
is known for them by \cite{MR1906155}.

\begin{theorem*}
  Let $S$ be a (possibly singular or generalised) del Pezzo surface of degree
  $d$, defined over a field $k$ of characteristic $0$. Then $S$ is an
  equivariant compactification of $\Ga^2$ over $k$ if and only if
  one of the following holds:
  \begin{itemize}
  \item $S$ has a non-singular $k$-rational point and is a form of $\Ptwo$,
    $\Pone \times \Pone$, the Hirzebruch surface $\FF_2$ or the corresponding
    singular del Pezzo surface,
  \item $S$ is a form of $\Bl_1\Ptwo$ or $\Bl_2\Ptwo$,
  \item $d = 7$ and $S$ is of type $\Aone$,
  \item $d = 6$ and $S$ is of type $\Aone$ (with 3 lines), $2\Aone$, $\Atwo$
    or $\Atwo+\Aone$,
  \item $d = 5$ and $S$ is of type $\Athree$ or $\Afour$,
  \item $d = 4$ and $S$ is of type $\Dfive$.
  \end{itemize}
\end{theorem*}

Table~\ref{tab:overview} summarises the results. For all del Pezzo
surfaces for which Manin's conjecture is known (at least in one
case), we have included references to the relevant articles.

In Lemma~\ref{lem:negative_curves}, we will give a criterion that will reduce
the number of ``candidates'' of generalised del Pezzo surfaces that might be
$\Ga^2$-varieties to a short list of surfaces that are connected by blow-ups
and blow-downs as presented in Figure~\ref{fig:blow-ups}.

Using a strategy described in Section~\ref{sec:strategy}, we will
show explicitly that the surfaces of type $\Aone$ in degree $6$,
type $\Athree$ in degree $5$ and type $\Dfive$ in degree $4$ are
$\Ga^2$-varieties, while type $\Dfour$ in degree $4$ and type
$\Esix$ in degree $3$ cannot have this structure. From these
``borderline cases'', some general considerations will allow us to
complete the classification over algebraically closed fields. Over non-closed
fields, some additional work will be necessary.

In Section~\ref{sec:ga_gm}, we will give an example of a del Pezzo
surface that is neither toric nor a $\Ga^2$-variety, but an
equivariant compactification of a semidirect product $\Ga \rtimes
\Gm$. This shows that it could be worthwhile even for del Pezzo
surfaces to extend the harmonic analysis approach to Manin's
conjecture to equivariant compactifications of more general
algebraic groups than tori and vector spaces.

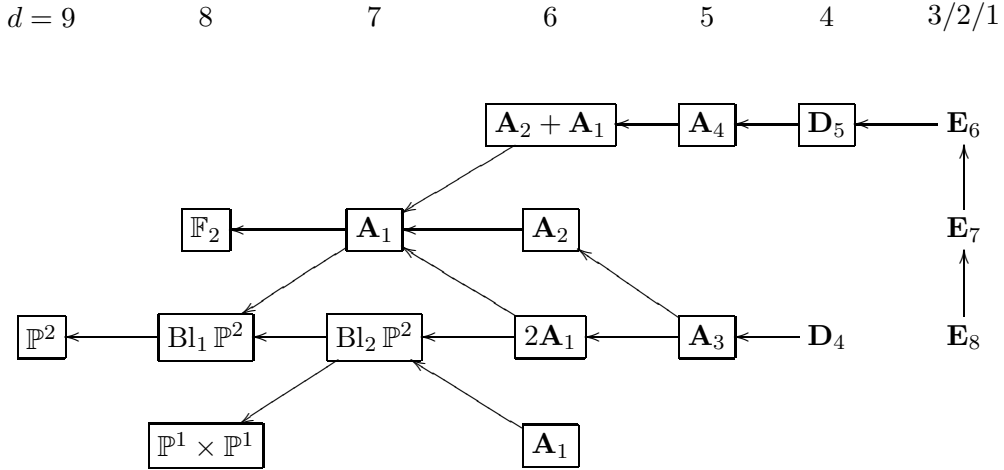
\begin{figure}[ht]
  \begin{equation*}
    \xymatrix{
      d=9 & 8 & 7 & 6 & 5 & 4 & 3/2/1\\
      & & & \ecbox{6}{\Atwo+\Aone} \ar@{->}[dl] & \ecbox{5}{\Afour} \ar@{->}[l] & \ecbox{4}{\Dfive} \ar@{->}[l] &
      \dpbox{3}{\Esix} \ar@{->}[l]\\
      & \ecbox{8}{\FF_2} & \ecbox{7}{\Aone} \ar@{->}[dl] \ar@{->}[l] &
      \ecbox{6}{\Atwo} \ar@{->}[l] & & & \dpbox{2}{\Eseven} \ar@{->}[u]\\
      \ecbox{9}{\Ptwo} & \ecbox{8}{\Bl_1\Ptwo} \ar@{->}[l] & \ecbox{7}{\Bl_2\Ptwo} \ar@{->}[l] \ar@{->}[dl]&
      \ecbox{6}{2\Aone} \ar@{->}[ul] \ar@{->}[l] & \ecbox{5}{\Athree}
      \ar@{->}[ul] \ar@{->}[l] & \dpbox{4}{\Dfour} \ar@{->}[l] & \dpbox{1}{\Eeight} \ar@{->}[u]\\
      & \ecbox{7}{\Pone \times \Pone} & & \ecbox{6}{\Aone} \ar@{->}[ul]
    }
  \end{equation*}
  \caption{Generalised del Pezzo surfaces $S$ defined over $\kbar$
    that satisfy $\#\{\text{negative curves on $S$}\} \le \rk\Pic(S)$. The
    boxed ones are equivariant compactifications of $\Ga^2$. Arrows denote
    blow-up maps.}
  \label{fig:blow-ups}
\end{figure}

\begin{table}[ht]
  \centering
  \begin{tabular}{|c||c|c||c|c||c|}
    \hline
    degree & type & lines & toric & $\Ga^2$-variety & Manin's
    conjecture\\
    \hline\hline
    9 & $\Ptwo$ & -- & yes & yes & \cite{MR1620682}, \cite{MR1906155} \\
    \hline
    8 & $\Bl_1\Ptwo$ & 1 & yes & yes & \cite{MR1620682}, \cite{MR1906155} \\
    & $\FF_2$ & -- & yes & yes & \cite{MR1620682}, \cite{MR1906155} \\
    \hline
    7 & $\Bl_2\Ptwo$ & 3 & yes & yes & \cite{MR1620682}, \cite{MR1906155} \\
    & $\Aone$ & 2 & yes & yes & \cite{MR1620682}, \cite{MR1906155} \\
    \hline
    6 & $\Bl_3\Ptwo$ & 6 & yes & -- & \cite{MR1620682} \\
    & $\Aone$ & 4 & yes & -- & \cite{MR1620682} \\
    & $\Aone$ & 3 & -- & yes & \cite{MR1906155} \\
    & $2\Aone$ & 2 & yes & yes & \cite{MR1620682}, \cite{MR1906155} \\
    & $\Atwo$ & 2 & -- & yes & \cite{MR1906155} \\
    & $\Atwo+\Aone$ & 1 & yes & yes & \cite{MR1620682}, \cite{MR1906155} \\
    \hline
    5 & $\Bl_4\Ptwo$ & 10 & -- & -- & \cite{MR1909606}, \cite{MR2099200}\\
    & $\Aone$ & 7 & -- & -- & -- \\
    & $2\Aone$ & 5 & yes & -- & \cite{MR1620682} \\
    & $\Atwo$ & 4 & -- & -- & \cite{arXiv:0710.1583}\\
    & $\Atwo+\Aone$ & 3 & yes & -- & \cite{MR1620682} \\
    & $\Athree$ & 2 & -- & yes & \cite{MR1906155} \\
    & $\Afour$ & 1 & -- & yes & \cite{MR1906155} \\
    \hline
    4 & $\Bl_5\Ptwo$ & 16 & -- & -- & \cite{arXiv:0808.1616}\\
    & $\Aone$ & 12 & -- & -- & --\\
    & $2\Aone$ & 9 & -- & -- & --\\
    & $2\Aone$ & 8 & -- & -- & \cite{arXiv:1002.0255}\\
    & $\Atwo$ & 8 & -- & -- & --\\
    & $3\Aone$ & 6 & -- & -- & --\\
    & $\Atwo+\Aone$ & 6 & -- & -- & --\\
    & $\Athree$ & 5 & -- & -- & \cite{derenthal}\\
    & $\Athree$ & 4 &-- & -- & --\\
    & $4\Aone$ & 4 & yes & -- & \cite{MR1620682}\\
    & $\Atwo+2\Aone$ & 4 & yes & -- & \cite{MR1620682} \\
    & $\Athree+\Aone$ & 3 & -- & -- & \cite{MR2520770}\\
    & $\Afour$ & 3 & -- & -- & \cite{MR2543667}\\
    & $\Dfour$ & 2 & -- & -- & \cite{MR2290499}\\
    & $\Athree + 2\Aone$ & 2 & yes & -- & \cite{MR1620682} \\
    & $\Dfive$ & 1 & -- & yes & \cite{MR1906155}, \cite{MR2320172} \\
    \hline
    3 & $\Dfive$ & 3 & -- & -- & \cite{MR2520769} \\
    & $3\Atwo$ & 3 & yes & -- & \cite{MR1620682}, \dots \\
    & $\Esix$ & 1 & -- & -- & \cite{MR2332351} \\
    & \dots & & & & \\
    \hline
    2 & $\Eseven$ & 1 & -- & -- & --\\
    & \dots & & & & \\
    \hline
    1 & $\Eeight$ & 1 & -- & -- & --\\
    & \dots & & & & \\
    \hline
  \end{tabular}
  \smallskip
  \caption{Singular del Pezzo surfaces over $\kbar$: all types of degree $\ge
    4$ and the relevant types of degree $\le 3$.}
  \label{tab:overview}
\end{table}

\medskip

\noindent\textbf{Acknowledgments:} This project was initiated during
the trimester program ``Diophantine equations'' at the Hausdorff
Research Institute for Mathematics (Bonn, Spring 2009). The authors
are grateful for the hospitality of this institution. The first
author was partially supported by grant DE~1646/1-1 of the Deutsche
Forschungsgemeinschaft, and the second author was funded by an EPSRC
student scholarship.

\section{Preliminaries}\label{sec:preliminary}

In this section, we start by recalling basic facts about the
structure and classification of del Pezzo surfaces and continue with
some elementary results on $\Ga^2$-varieties under blow-ups.
We work over a field $k$ of characteristic $0$ with algebraic closure $\kbar$.

\medskip

For $n \in \{1,2\}$, a \emph{$(-n)$-curve} on a non-singular
projective surface is a smooth rational curve defined over $\kbar$
with self-intersection number $-n$.  Over $\kbar$, every generalised
del Pezzo surface $S$ can be realised as either $\Ptwo$, $\Pone
\times \Pone$, the Hirzebruch surface $\FF_2$ or a \emph{blow-up of
$\Ptwo$ in $r \le 8$ points in almost general position}, which means
that $S$ is obtained from $\Ptwo$ by a series of $r \le 8$ maps
\begin{equation*}
  S=S_r \to S_{r-1} \to \dots \to S_1 \to S_0=\Ptwo
\end{equation*}
where each map $S_i \to S_{i-1}$ is the blow-up of a point not lying
on a $(-2)$-curve of $S_{i-1}$. The \emph{degree} of $S$ is the
self-intersection number of its anticanonical class $-K_S$; it is
$9-r$ in the case of blow-ups of $\Ptwo$ in $r \le 8$ points. A
generalised del Pezzo surface $S$ is ordinary if and only if it does
not contain $(-2)$-curves; this is true for $\Ptwo$, $\Pone \times
\Pone$ and blow-ups of $\Ptwo$ in $r \le 8$ points \emph{in general position}
(see \cite[Th\'eor\`eme~III.1]{MR579026}, for example).

In each degree, we say that two del Pezzo surfaces have the same \emph{type}
if their \emph{extended Dynkin diagrams} (the dual graphs of their
configurations of negative curves over $\kbar$) coincide. In general, there
are several isomorphism classes of del Pezzo surfaces of the same type (e.g.,
infinite families of ordinary del Pezzo surfaces in degree $\le 4$), but over
$\kbar$ in all the cases that we will be interested in, each surface is
uniquely determined by its type. In each degree, we will label the types by
the connected components of $(-2)$-curves in their extended Dynkin diagrams
(in the $\mathbf{ADE}$-notation); in many cases, this determines the type
uniquely, but sometimes, ones must additionally mention the number of
$(-1)$-curves (e.g., type $\Aone$ in degree $6$ with $3$ or $4$
$(-1)$-curves).

Classifying singular del Pezzo surfaces according to their degree,
the types of their singularities and, if necessary, their number of
lines gives the same result. See \cite{MR579026}, \cite{MR80f:14021},
\cite{MR89f:11083} or \cite{MR2227002} for further details.

A surface $S$ defined over $k$ is a (ordinary, generalised or singular) del
Pezzo surface if $S_\kbar = S \times_k \kbar$ has such a structure over the
algebraic closure $\kbar$; by definition, the type of $S$ is the type of
$S_\kbar$. We say that $S$ is a \emph{form} of $S'$ if $S_\kbar$ and
$S'_\kbar$ are isomorphic. A generalised (resp.\ singular) del Pezzo surface
defined over $k$ is called \emph{split} if it (resp.\ its minimal
desingularisation) is isomorphic over $k$ to $\Ptwo$, $\Pone \times \Pone$,
$\FF_2$ or a blow-up of $\Ptwo$ in $k$-rational points.

\medskip

If $\GG$ is a connected linear algebraic group defined over $k$,
then we say that a proper variety $V$ defined over $k$ is an
\emph{equivariant compactification of $\GG$ over $k$} or
alternatively a \emph{$\GG$-variety over $k$}, if $\GG$ acts on $V$,
with the action being defined over $k$, and there exists an open
subset $U \subset V$ which is \emph{equivariantly} isomorphic to
$\GG$ over $k$. By an equivariant morphism, we mean a morphism
commuting with the action of $\GG$. We note that any algebraic group
over $k$ which is isomorphic to $\Ga^n$ over $\kbar$, is also
isomorphic to $\Ga^n$ over $k$.

An \emph{equivalence} between $\GG$-varieties $X_1, X_2$ is a
commutative diagram
\begin{equation}\label{eq:equivalence}
  \begin{split}
    \xymatrix{
      \GG \times X_1 \ar[d] \ar[r]^{(\alpha,j)} & \GG \times X_2 \ar[d] \\
      X_1 \ar[r]^j & X_2 }
  \end{split}
\end{equation}
where $\alpha: \GG \to \GG$ is an automorphism and $j : X_1 \to X_2$
is an isomorphism.

\begin{lemma}\label{lem:P2}
  Up to equivalence, there are precisely two distinct $\Ga^2$-structures on
  $\Ptwo$ over $\kbar$.  They are given by the following representations of
  $\Ga^2$:

  \begin{equation*}
    \tau(a,b)=\left( \begin{array}{ccc}
        1 & 0 & 0 \\
        a & 1 & 0 \\
        b & 0 & 1 \end{array} \right), \quad
    \rho(a,b)=\left( \begin{array}{ccc}
        1 & 0 & 0 \\
        a & 1 & 0 \\
        b+\frac{1}{2}a^2 & a & 1 \end{array} \right).
  \end{equation*}
\end{lemma}

\begin{proof}
    See \cite[Proposition~3.2]{MR1731473}.
\end{proof}

\begin{lemma}\label{lem:blow_down}
  Let $S$ be a non-singular $\Ga^2$-variety over $k$, and $E \subset S$ a
  $(-1)$-curve which is invariant under the action of the Galois group
  $\Gal(\kbar/k)$. Then there exists a $\Ga^2$-equivariant $k$-morphism that
  blows down $E$.
\end{lemma}

\begin{proof}
  See \cite[Proposition~5.1]{MR1731473} for the corresponding statement over
  $\kbar$. It is clear that if $E$ is invariant under the action of the Galois
  group $\Gal(\kbar/k)$, then the corresponding morphism is defined over $k$.
\end{proof}

\begin{lemma}\label{lem:blow_up}
  Let $\GG$ be a connected linear algebraic group over $k$, and let $S$ be a
  projective surface which is a $\GG$-variety over $k$.  Let $\pi: \tS \to S$
  be the blow-up of $S$ at a collection of distinct points defined over
  $\kbar$ that are invariant under the action of $\GG$ and conjugate under the
  action of the Galois group $\Gal(\kbar/k)$. Then $\tS$ can be equipped with
  a $\GG$-structure over $k$ in such a way that $\pi:\tS \to S$ is a
  $\GG$-equivariant $k$-morphism.
\end{lemma}

\begin{proof}
  It is clear that the blow-up of conjugate points is defined over $k$. Thus
  it suffices to show that this morphism is also $\GG$-equivariant.

  Let $E$ be the exceptional divisor of the blow-up. Then applying the
  universal property of blow-ups \cite[Corollary~II.7.15]{MR0463157} to the
  natural $k$-morphism $f: \GG \times S \to S$, we see that there exists a
  $k$-morphism $\widetilde{f}$ such that the following diagram commutes.
  \begin{equation*}
    \xymatrix{\GG \times \tS \ar[d]_{(\mathrm{id},\pi)} \ar[r]^{\widetilde{f}}
      & \tS \ar[d]^{\pi} \\ \GG \times S \ar[r]^f & S }
  \end{equation*}
  A priori, we only know that the map $\widetilde{f}$ satisfies the identities
  $ex=x$ and $(gh)^{-1}g(h(x))=x$ for all $g,h \in \GG$ and $x \in
  \tS\setminus E$. However any morphism which is equal to the identity on an
  open dense subset of $\tS$ must also be equal to the identity on all of
  $\tS$. That is, these identities do in fact hold on all of $\tS$ and we get
  an action of $\GG$ on $\tS$ over $k$.
\end{proof}

\begin{lemma}\label{lem:singular}
    Let $S$ be a singular del Pezzo surface over $k$, and $\tS$ its
    minimal desingularisation. Then $S$ is a $\Ga^2$-variety over $k$ if and
    only if $\tS$ is.
\end{lemma}

\begin{proof}
  Suppose $S$ is a $\Ga^2$-variety over $k$. Since $\Ga^2$ is connected, the
  orbit of a singularity under this action is connected as well. Furthermore,
  every point in the orbit is a singularity as well (since translation by an
  element of $\Ga^2$ is an isomorphism). But there is only a finite number of
  (isolated) singularities. Therefore, the orbit is just one point, so that
  each singularity is fixed under the $\Ga^2$-action. By a similar argument,
  we see that the Galois group $\Gal(\kbar/k)$ at worst swaps any
  singularities.  Hence we can resolve the singularities via blow-ups and
  applying Lemma~\ref{lem:blow_up}, we see that $\tS$ is also a
  $\Ga^2$-variety over $k$.

  Next, suppose that $\tS$ is a $\Ga^2$-variety over $k$. The anticanonical
  class is defined over $k$, and hence the anticanonical map (or a multiple of
  it in degrees $1$ and $2$) is defined over $k$ and contracts precisely the
  $(-2)$-curves, so that its image is the corresponding singular del Pezzo
  surface $S$. This map is $\Ga^2$-equivariant by
  \cite[Proposition~2.3]{MR1731473} and \cite[Corollary~2.4]{MR1731473}.
\end{proof}

\begin{lemma}\label{lem:negative_curves}
  If a generalised del Pezzo surface $\tS$ is an equivariant compactification
  of $\Ga^2$ over $k$, then the number of negative curves contained in
  $\tS_\kbar$ is at most the rank of $\Pic(\tS_\kbar)$.
\end{lemma}

\begin{proof}
  As explained in \cite[Section~2.1]{MR1731473}, the complement of the open
  $\Ga^2$-orbit on $\tS_\kbar$ is a divisor, called the boundary divisor.  By
  \cite[Proposition~2.3]{MR1731473}, $\Ga^2$ acts trivially on
  $\Pic(\tS_\kbar)$, and since any negative curve is the unique effective
  divisor in its divisor class, $\Ga^2$ must fix each negative curve (not
  necessarily pointwise). Therefore, negative curves must be components of the
  boundary divisor. By \cite[Theorem~2.5]{MR1731473}, the Picard group of
  $\tS_\kbar$ is \emph{freely} generated by its irreducible components, and
  the result follows.
\end{proof}

\section{Strategy}\label{sec:strategy}

In the proof of our main result, we will show explicitly whether certain
singular del Pezzo surfaces are $\Ga^2$-varieties. We use the following
strategy.  In this section, we work over an algebraically closed field $\kbar$
of characteristic~$0$.

Let $i: S \inj \PP^d$ be an anticanonically embedded singular del
Pezzo surface of degree $d \in \{3, \dots, 7\}$, and let $\pi_0: \tS
\to S$ be its minimal desingularisation, which is also the blow-up
$\pi_1: \tS \to \Ptwo$ of $\Ptwo$ in $r = 9-d$ points in almost
general position. We have the diagram
\begin{equation}\label{eq:blow-ups}
  \begin{split}
    \xymatrix{\tS \ar[rrd]^{\pi_1} \ar[d]_{\pi_0} & & \\
      S \ar[r]_i& \PP^d \ar@{-->}[r]_{\pi_2} & \Ptwo
      \ar@/^1.5pc/@{-->}[ll]^\phi}
  \end{split}
\end{equation}
where $\pi_2: \PP^d \rto \Ptwo$ is the projection to a plane in
$\PP^d$ and $\phi: \Ptwo \rto S$ is the inverse of $\pi_2 \circ i$,
given by a linear system of cubics $V \subset H^0(\Ptwo,
\O_\Ptwo(3))$.

If $S$ is a $\Ga^2$-variety, this induces $\Ga^2$-structures on
$\tS$ and $\Ptwo$, by Lemma~\ref{lem:singular} and
Lemma~\ref{lem:blow_down}; in other words, any $\Ga^2$-structure on
$S$ is induced by a $\Ga^2$-structure on $\Ptwo$. To find a
$\Ga^2$-structure on $S$ or to prove that it does not exist, we
would like to test whether one of the $\Ga^2$-structures on $\Ptwo$
induces a $\Ga^2$-structure on $S$. This is done by checking whether
or not the linear system $V$ is invariant under the uniquely
determined induced $\Ga^2$-action on $H^0(\Ptwo, \O_\Ptwo(3))$ (see
\cite[Proposition~2.3]{MR1731473}). Note that it is not enough to
check whether the base points of $V$ are fixed under this action.

By Lemma~\ref{lem:P2}, there are only two equivalence classes of
$\Ga^2$-structures on $\Ptwo$. A priori, however, one might have to
test not one, but every $\Ga^2$-structure in each equivalence class.

Fortunately, we can simplify the task as follows. For the del Pezzo
surfaces that we are interested in, the number of negative curves on
$\tS$ is $\rk \Pic(\tS) = r+1$. Indeed, this follows from
Lemma~\ref{lem:negative_curves} and the fact that the cone of
effective divisors in $\Pic(\tS) \otimes_\ZZ \RR \cong \RR^{r+1}$ is
full-dimensional and generated by negative curves for $d \le 7$ by
\cite[Theorem~3.10]{MR2377367}. Under the map $\pi_1: \tS \to
\Ptwo$, one negative curve is mapped to a line $\ell \subset \Ptwo$,
while the other $r$ negative curves are projected to (one or more)
points $p_1, \dots, p_n$ on $\ell$.

As explained in the proof of Lemma~\ref{lem:negative_curves}, any
$\Ga^2$-structure on $\tS$ fixes the negative curves (not
necessarily pointwise). Therefore, any $\Ga^2$-structure on $\Ptwo$
that induces a $\Ga^2$-structure on $S$ and $\tS$ must fix $\ell$
and $p_1, \dots, p_n$.

This restricts the $\Ga^2$-structures on $\Ptwo$ that we must
consider in each of the two equivalence classes of $\tau, \rho$
described in Lemma~\ref{lem:P2}. Let us work this out explicitly, in
coordinates $x_0,x_1,x_2$ on $\Ptwo$ such that $\ell = \{x_0=0\}$
and $p_1=(0:0:1)$.
\begin{itemize}
\item \emph{$\Ga^2$-structures equivalent to $\tau$:} Consider the
  diagram~\eqref{eq:equivalence} where $X_1$ is $\Ptwo$ with the standard
  structure $\tau$, and $X_2$ is $\Ptwo$ with an equivalent structure $\tau'$.
  The diagram is commutative if and only if
  \begin{equation*}
    \tau'(\alpha(a,b))\xx = j(\tau(a,b)(j^{-1}(\xx)))
  \end{equation*}
  for any $(a,b) \in \Ga^2$ and $\xx \in \Ptwo$.
  The isomorphism $j: X_1 \to X_2$ is given by a matrix $A \in
  \PGL_3(\kbar)$ that must be of the form
  \begin{equation*}
    A =
    \begin{pmatrix}
      1 & 0 & 0 \\ a_{10} & a_{11} & a_{12} \\ a_{20} & a_{21} & a_{22}
    \end{pmatrix}
  \end{equation*}
  since it must map the line fixed by $\tau$ to $\ell$. It is now
  straightforward to compute that
  \begin{equation*}
    \tau'(\alpha(a,b))=
    \begin{pmatrix}
      1 & 0 & 0 \\
      a_{11} a + a_{12} b & 1 &  0\\
      a_{21} a + a_{22} b & 0 &  1
    \end{pmatrix}.
  \end{equation*}
  Since $\alpha$ is an automorphism of $\Ga^2$ and the lower right $2 \times
  2$-submatrix of $A$ is invertible, the linear series $V$ defining $\phi :
  \Ptwo \rto S$ is invariant under $\tau'$ if and only if it is invariant
  under the standard structure $\tau$.
\item \emph{$\Ga^2$-structures equivalent to $\rho$:} We argue as in the case
  of $\tau$. Since $\rho$ fixes a line $\{x_0=0\}$, but only one point
  $(0:0:1)$ on it, a structure $\rho'$ equivalent to $\rho$ might induce an
  action on $S$ only if $\pi_1$ maps the negative curves on $\tS$ to $\ell$
  fixed by $\rho'$ and one point $p_1$ fixed by $\rho'$. Therefore, $\tS$ must
  be the blow-up of precisely one point in $\Ptwo$ and further points on the
  exceptional divisors.

  This also further restricts the shape of the matrix of $j$. Computing the
  matrix of $\rho'(\alpha(a,b))$ is now straightforward. We omit it here, but
  remark that it is in general unclear whether testing the linear series $V$
  defining $\phi : \Ptwo \rto S$ for invariance under $\rho$ is enough -- we
  might have to consider all equivalent $\rho'$, using the matrices that we
  just computed.

  However, in our applications the following fact will be sufficient: the
  matrix of $\rho'(\alpha(a,b))$ is a lower triangular matrix with ``1''s on
  the diagonal and the property that, for any choice of $j$, its entries below
  the diagonal are non-zero for general $(a,b) \in \Ga^2$.
\end{itemize}

\section{Proof of the main result}

Here, $k$ is a field of characteristic $0$ with algebraic closure $\kbar$.  By
Lemma~\ref{lem:singular}, we can interchange freely between a singular del
Pezzo surface and its minimal desingularisation.

We apply Lemma~\ref{lem:negative_curves} and extract those
generalised del Pezzo surfaces $S$ whose number of negative curves
is at most the rank of  $\Pic(S_{\kbar})$ from the classification of
generalised del Pezzo surfaces that can be found in
\cite{MR80f:14021}, \cite{MR89f:11083}, \cite{MR2227002} (see
\cite[Tables~2--5]{math.AG/0604194} for a summary of the data
relevant to us). This leaves the $16$ types of surfaces of degrees
$1$ to $9$ that can be found in Figure~\ref{fig:blow-ups}, together
with various blow-up maps between them.

Note that, over $\kbar$, all of them except the degree $1$ del Pezzo
surface of type $\Eeight$ (which has two isomorphism classes by
\cite[Lemma~4.2]{MR1933881}) are unique up to isomorphism. Indeed,
this is true for type $\Aone$ of degree $6$ with $3$ lines because
its minimal desingularisation is the blow-up of $\Ptwo$ in three
points on one line, which are clearly unique up to automorphism of
$\Ptwo$; a similar argument applies to all cases of degree $\ge 7$.
Uniqueness is known for type $\Eseven$ of degree $2$ by
\cite[Lemma~4.6]{MR1933881}. For types $\Esix$ and $\Dfive$ of
degree $3$, uniqueness was proved in \cite{MR80f:14021}, and all
remaining del Pezzo surfaces of degree $4$, $5$ and $6$ are obtained
from the desingularisations of these two cubic surfaces by
contracting certain $(-1)$-curves, which implies that they are also
unique (for type $\Athree$ of degree $5$, which can be obtained from
type $\Dfour$ of degree $4$ in two ways, we observe additionally
that there is an automorphism of the quartic del Pezzo surface with
$\Dfour$ singularity which swaps the two lines).

Over $k$, the split generalised del Pezzo surfaces of degree $\ge 3$ in
question are unique up to isomorphism. Indeed, for the cubic surface $S$ of
type $\Esix$ (resp.\ $\Dfive$), \cite[Theorem~3]{sakamaki} (stated over $\CC$,
but the proof works over any algebraically closed field of characteristic $0$)
determines the automorphism group $\Aut(S_\kbar)$ as $\kbar \rtimes \kbar^*$
(resp.\ $\kbar^*$), hence $H^1(\Gal(\kbar/k), \Aut(S_\kbar))$ is trivial and
$S$ has no non-trivial forms over $k$.  For the remaining types of degree $\ge
4$, uniqueness follows as before.

Using the strategy described in Section~\ref{sec:strategy}, we show
that the following three surfaces are $\Ga^2$-varieties by
describing a $\Ga^2$-action explicitly.

\begin{lemma}
  The following split singular del Pezzo surfaces are $\Ga^2$-varieties:
    \begin{itemize}
        \item type $\Dfive$ of degree $4$,
        \item type $\Athree$ of degree $5$,
        \item type $\Aone$ of degree $6$ (with $3$ lines).
    \end{itemize}
\end{lemma}

\begin{proof}
  We treat each case individually and use the notation of
  diagram~(\ref{eq:blow-ups}).
  \begin{itemize}
  \item \emph{$\Dfive$ of degree $4$:} An anticanonical embedding $i: S \inj
    \PP^4$ of this singular del Pezzo surface is:
    \begin{equation*}
      S:x_0x_1-x_2^2= x_0x_4-x_1x_2+x_3^2=0.
    \end{equation*}
    A birational map to $\Ptwo$ is given via the projection $\pi_2:\PP^4 \rto
    \Ptwo$ defined by $\xx \mapsto (x_0:x_2:x_3)$. The image of one of the
    $(-2)$-curves on the minimal desingularisation $\pi_0: \tS \to S$ under
    $\pi_1: \tS \to \Ptwo$ is $\ell=\{x_0=0\}$.

    As explained in Section~\ref{sec:strategy}, in this situation, the only
    $\Ga^2$-structure on $\Ptwo$ in the equivalence class of $\tau$
    (cf. Lemma~\ref{lem:P2}) that might induce an action on $S$ is the
    structure $\tau$ itself.

    We compute the induced action on $S$ via the inverse
    \begin{align*}
      \phi : \Ptwo &\rto S \\
      (x_0:x_2:x_3) &\mapsto
      (x_0^3:x_0x_2^2:x_0^2x_2:x_0^2x_3:x_2^3-x_0x_3^2)
    \end{align*}
    of $\pi_2 \circ i$. For $(a,b) \in \Ga^2$, it is given by
    \begin{equation*}
      \begin{pmatrix}
        1       & 0   & 0     & 0  & 0 \\
        a^2     & 1   & 2a    & 0  & 0 \\
        a       & 0   & 1     & 0  & 0 \\
        b       & 0   & 0     & 1  & 0 \\
        b^2-a^3 & -3a & -3a^2 & 2b & 1 \\
      \end{pmatrix}.
    \end{equation*}
    It is easy enough to check that $S$ is invariant under this.

    We note that the action on the line $\{x_0=x_2=x_3=0\}$ in $S$ is
    non-trivial, with the fixed point being the singularity of $S$.  So there
    is no hope of blowing up a point on this surface to create another
    equivariant compactification of $\Ga^2$ of degree $3$ from this structure.
  \item \emph{$\Athree$ of degree $5$:} In the model
    \begin{equation*}
      \begin{split}
        S: {}&x_0x_2-x_1^2 = x_0x_3-x_1x_4 = x_2x_4-x_1x_3\\
        ={}&x_2x_4+x_4^2+x_0x_5 = x_2x_3+x_3x_4+x_1x_5 = 0
      \end{split}
    \end{equation*}
    given in \cite[Section~6]{math.AG/0604194}, we can choose $\pi_2$ as $\xx
    \mapsto (x_0:x_1:x_4)$. Then $\pi_1$ maps one of the $(-2)$-curves to
    $\ell=\{x_0=0\}$. This motivates us to consider the action on $\PP^5$
    induced by $\tau$ on $\Ptwo$ that is given by the representation
    \begin{equation*}
      \begin{pmatrix}
        1         & 0    & 0  & 0   & 0       & 0 \\
        a         & 1    & 0  & 0   & 0       & 0 \\
        a^2       & 2a   & 1  & 0   & 0       & 0 \\
        ab        & b    & 0  & 1   & a       & 0 \\
        b         & 0    & 0  & 0   & 1       & 0 \\
        -a^2b-b^2 & -2ab & -b & -2a & -a^2-2b & 1 \\
      \end{pmatrix}.
    \end{equation*}
    One easily checks that it fixes $S$.
  \item \emph{$\Aone$ of degree $6$ (with $3$ lines):} This surface is the
    blow-up of three points on the line at infinity in $\Ptwo$.  However, the
    action of $\tau$ on $\Ptwo$ fixes this line. Then a simple application of
    Lemma~\ref{lem:blow_up} shows that this surface is a $\Ga^2$-variety.
  \end{itemize}
  This completes the proof of the lemma.
\end{proof}

Since these three split singular del Pezzo surfaces are $\Ga^2$-varieties, the
same holds for the corresponding split generalised del Pezzo surfaces.
Contracting the $(-1)$-curves and using Lemma~\ref{lem:blow_down}, all other
split generalised del Pezzo surfaces marked by a box in
Figure~\ref{fig:blow-ups} are $\Ga^2$-varieties, and the same holds for the
corresponding split singular del Pezzo surfaces.

We now need to determine $\Ga^2$-structures on the corresponding
non-split surfaces. Our task is made easier by the fact that many of
the surfaces under consideration are automatically split.

\begin{lemma}\label{lem:always_split}
  Any form of $\Ptwo$ or $\FF_2$ with a $k$-rational point is split.
  Moreover, any form of $\Bl_1\Ptwo$ and any generalised del Pezzo
  surface with degree $d=7$ of type $\Aone$, $d=6$ of type
  $\Atwo+\Aone$ or $2\Aone$, $d=5$ of type $\Afour$ or $\Athree$ or
  $d=4$ of type $\Dfive$ is split.
\end{lemma}

\begin{proof}
  It is a classical result that any form of $\Ptwo$ with a $k$-rational point
  is split.

  The unique $(-1)$-curve on a form $S$ of $\Bl_1\Ptwo$ is defined over
  $k$. Its contraction gives a form of $\Ptwo$ with a $k$-rational point (the
  image of the $(-1)$-curve), so that this form is $\Ptwo$ itself, and $S$ is
  the blow-up of $\Ptwo$ in a $k$-rational point.

  For the cases of degree $\le 7$, we note that their extended Dynkin diagrams
  (which can be found in \cite[Section~6 and 8]{MR89f:11083}, for example)
  have no symmetry, so that all their negative curves are defined over
  $k$. Therefore, these surfaces are obtained from $\Ptwo$ by a series of
  blow-ups of $k$-rational points.

  Finally, let $S$ be a form of $\FF_2$ containing a $k$-rational point
  $p$. If $p$ does not lie on the unique $(-2)$-curve $B$ in $S$, then blowing
  up $p$ gives a surface $S'$ of degree $7$ and type $\Aone$. So $S$ is
  obtained from $S'$ by contracting a certain $(-1)$-curve. As $S'$ is split
  and unique up to $k$-isomorphism, the same is true for $S$, which is
  therefore $k$-isomorphic to $\FF_2$. If $p$ does lie on $B$ in $S$, then the
  fibre $F$ through $p$ is uniquely determined and hence defined over
  $k$. Therefore $F$ is isomorphic to $\Pone$ over $k$, and so contains a
  $k$-rational point not lying on $B$.
\end{proof}

To complete the proof of one direction of our theorem, it remains to exhibit
the structure of a $\Ga^2$-variety in the following cases of generalised del
Pezzo surfaces $S$ defined over $k$:
\begin{itemize}
\item A form of $\Bl_2\Ptwo$: Contracting the two (possibly conjugate)
  non-intersecting $(-1)$-curves gives a form $S'$ of $\Ptwo$ with a line (the
  image of the third $(-1)$-curve on $S$) defined over $k$, so that $S'$ is
  split.  We equip it with a $\Ga^2$-structure fixing the line. Therefore, $S$
  is the blow-up of $\Ptwo$ in a collection of two (possibly conjugate) points
  on a line fixed by the $\Ga^2$-action, which is a $\Ga^2$-variety over $k$
  by Lemma~\ref{lem:blow_up}.
\item A form of $\Pone \times \Pone$ with a $k$-rational point $p$: Blowing up
  $p$ gives a form $S'$ of $\Bl_2\Ptwo$. As above, the surface $S'$ is a
  $\Ga^2$-variety over $k$, and, by Lemma~\ref{lem:blow_down}, the same is
  true for $S$.
\item A form of the degree $6$ surface of type $\Aone$: We argue as in the
  case $\Bl_2\Ptwo$, and see that this surface is the blow-up of $\Ptwo$ at
  three (possibly conjugate) points on a line defined over $k$, so is a
  $\Ga^2$-variety over $k$.
\item A form of the degree $6$ surface of type $\Atwo$: Contracting the two
  (possibly conjugate) $(-1)$-curves on $S$ gives a form $S'$ of $\FF_2$ with
  two (possibly conjugate) points on the same fibre $F$; this fibre is defined
  over $k$. Arguing as in the proof of Lemma~\ref{lem:always_split}, $S'$ is
  split. It suffices to show that there exists a $\Ga^2$-structure on $S'$
  over $k$ which fixes $F$ pointwise, since then we can then apply
  Lemma~\ref{lem:blow_up} to get the required action on $S$.

  Such a $\Ga^2$-structure can be found by blowing up a $k$-point on $F$
  outside the unique (-2)-curve $B$. This gives a surface of degree $7$ and
  type $\Aone$ with an exceptional curve $E$ defined over $k$. We equip this
  surface with the structure of a $\Ga^2$-variety over $k$ induced from the
  first action on $\Ptwo$ described in Lemma~\ref{lem:P2}. Here the strict
  transform $\widetilde{F}$ of $F$ is equal to the strict transform of the
  line fixed pointwise in $\Ptwo$, thus $F$ is also fixed pointwise and we get
  the required action on $S'$.
\end{itemize}

Finally, we must show that the remaining del Pezzo surfaces given in
Figure~\ref{fig:blow-ups} are \emph{not} equivariant compactifications of
$\Ga^2$.

\begin{lemma}\label{lem:not_Ga2_variety}
  The following del Pezzo surfaces are not equivariant compactifications of
  $\Ga^2$:
  \begin{itemize}
  \item forms of $\Ptwo$, $\Pone \times \Pone$ and $\FF_2$ without
    $k$-rational points,
  \item type $\Esix$ of degree $3$,
  \item type $\Dfour$ and degree $4$.
  \end{itemize}
\end{lemma}

\begin{proof}
  As any $\Ga^2$-variety over $k$ contains an open subset isomorphic to
  $\Ga^2$ over $k$, it must contain a $k$-rational point.

  For the remaining two surfaces, it is enough to work over $\kbar$.  To prove
  that a generalised del Pezzo surface $\tS$ is \emph{not} a $\Ga^2$-variety,
  we use the startegy and notation of Section~\ref{sec:strategy} again
  (cf. \cite[Remark~3.3]{MR2029868}).
\begin{itemize}
\item \emph{$\Esix$ of degree $3$:}
  We consider the anticanonical embedding $i: S \inj \PP^3$ defined by
  \begin{equation*}
    S : x_1x_0^2+x_0x_3^2+x_2^3=0,
  \end{equation*}
  and $\pi_2 : \xx \mapsto (x_0:x_2:x_3)$. Then $\phi$ is given
  by
  \begin{equation*}
    (x_0:x_2:x_3) \mapsto (x_0^3:-(x_0x_3^2+x_2^3):x_0^2x_2:x_0^2x_3).
  \end{equation*}

  Since $\pi_1$ maps one of the $(-2)$-curves on $\tS$ to $\ell = \{x_0 = 0\}$
  and all other negative curves to $p_1=(0:0:1)$, we must show that the linear
  series defining $\phi$ is neither invariant under the $\Ga^2$-action induced
  by $\tau$ nor under one of the actions described in
  Section~\ref{sec:strategy} that are equivalent to $\rho$.

  For the relevant actions $\rho'$ equivalent to $\rho$, it is straightforward
  to check (only using the facts about the lower triangular representations of
  $\rho'$ stated at the end of Section~\ref{sec:strategy}) that the linear
  series cannot be invariant.  For $\tau$, see \cite[Remark~3.3]{MR2029868}.
\item \emph{$\Dfour$ of degree $4$:} Similarly, assume that $S$ of type
  $\Dfour$ and degree $4$ is a $\Ga^2$-variety; see
  \cite[Lemma~2.1]{MR2290499} for its equation and geometric properties. By
  \cite[Lemma~2.2]{MR2290499}, the negative curves on its minimal
  desingularisation $\tS$ are mapped by $\pi_1$ to a line $\ell \subset \Ptwo$
  and two distinct points $p_1,p_2$ on it. As explained in
  Section~\ref{sec:strategy}, this rules out a $\Ga^2$-structure induced by a
  structure on $\Ptwo$ equivalent to $\rho$. Finally, see
  \cite[Lemma~2.3]{MR2290499} for a proof that $S$ does not have a
  $\Ga^2$-structure induced by $\tau$.
\end{itemize}
This completes the proof of the lemma.
\end{proof}

Finally, we note that if the generalised del Pezzo surfaces of type
$\Eseven$ of degree $2$ or type $\Eeight$ of degree $1$ were
$\Ga^2$-varieties, the same would hold for type $\Esix$ of degree
$3$ (by contracting $(-1)$-curves, see Lemma~\ref{lem:blow_down}),
contradicting Lemma~\ref{lem:not_Ga2_variety}.

Thus we have shown that the list given in the statement of our
theorem is complete.

\section{An equivariant compactification of $\Ga \rtimes
  \Gm$}\label{sec:ga_gm}

Let $S$ be the singular quartic del Pezzo surface of type
$\Athree+\Aone$ defined by
\begin{equation*}
  S : x_0^2+x_0x_3+x_2x_4=x_1x_3-x_2^2=0.
\end{equation*}
In this section, we show that this is an example of a del Pezzo
surface that is an equivariant compactification of a semidirect
product of $\Ga$ and $\Gm$, but is neither toric nor a
$\Ga^2$-variety. Manin's conjecture has been proved for this surface
in \cite[Section~8]{MR2520770}, not by exploiting this
structure, but using the universal torsor method.

\medskip

The singularities on $S$ are $(0:0:0:0:1)$ of type $\Athree$ and
$(0:1:0:0:0)$ of type $\Aone$. It contains three lines
$\{x_0=x_1=x_2=0\}$, $\{x_0+x_3=x_1=x_2=0\}$, $\{x_0=x_2=x_3=0\}$.

The projection $\xx \mapsto (x_0:x_1:x_2)$ from the first line is a
birational map $S \rto \Ptwo$, with inverse $\Ptwo \rto S$ defined
by
\begin{equation*}
  (y_0:y_1:y_2) \mapsto (y_0y_1y_2 : y_1^2y_2 : y_1y_2^2 : y_2^3 :
  -y_0(y_2^2+y_0y_1)).
\end{equation*}
These birational maps induce isomorphisms between the complement $U$
of the lines on $S$ and $U'=\{y_1y_2 \ne 0\} \subset \Ptwo$.

\medskip

Let $\Ga \rtimes \Gm$ be the semidirect product of $\Ga$ and $\Gm$
via $\phi: \Gm \to \Aut(\Ga)$ defined by $\phi_t(b)=t^{-1}b$ for $t
\in \Gm$ and $b \in \Ga$.

The action of $(b,t) \in \Ga \rtimes \Gm$ on $S$ is given by the
representation
\begin{equation*}
  \begin{pmatrix}
    1   & 0   & bt    & 0  & 0      \\
    0   & t^2 & 0     & 0  & 0      \\
    0   & 0   & t     & 0  & 0      \\
    0   & 0   & 0     & 1  & 0      \\
    -2b & 0   & -tb^2 & -b & t^{-1}
  \end{pmatrix}.
\end{equation*}

Its only fixed points are the singularities (so there is no hope to
produce from this example a singular cubic surface that is an
equivariant compactification of $\Ga \rtimes \Gm$).

The $\Ga \rtimes \Gm$-action on $S$ described above is induced by
the action on $\Ptwo$ defined by
\begin{equation*}
  \begin{pmatrix}
    t^{-1} & 0  & b \\
    0      & t  & 0 \\
    0      & 0  & 1
  \end{pmatrix}.
\end{equation*}

The open orbit under the $\Ga \rtimes \Gm$-action is the complement
$U$ of the lines on $S$ (resp. $U' \subset \Ptwo$).

\bibliographystyle{alpha}

\bibliography{compact}

\end{document}